\documentclass[12pt]{amsart}
\font\smc=cmcsc10 scaled \magstep1
 
\newtheorem{Theorem}{Theorem} 
\newcommand\noi{\noindent}
\newcommand\sep{,\ }  
\newcommand\darw{\rightleftharpoons}

\newcommand\PP{\mathbb P}

\newcommand\NN{\mathbb N}
\newcommand\TT{\mathbb T}
\newcommand\ZZ{\mathbb Z}

\renewcommand\phi{\varphi}
\def\capt#1;#2{Table #1. \ #2\medskip}    

\begin{document}
\title{Inverses of monomial Cremona transformations}
\footnotetext[1]{\rm MSC-class: 14E07 (Primary); 15A29, 11C20 (Secondary).}
\footnotetext[2]{{\smc Keywords:}
{\rm monomial Cremona transformation\sep birational morphism\sep degree of inverse.}}
\author{Peter M. Johnson}
\address{}  

\begin{abstract}
\noindent We show that monomial Cremona transformations of degree $d$ 
in $\PP^{n}$ can have inverses whose degree $d'$ is quite large
(for $d > 2$, $d' = \frac{(d-1)^{n}-1}{d-2}$ occurs), and that the full list
of possible degrees $d'$ for fixed $d$ and $n$ does not always form an interval.
An easy method for inverting the maps is presented.
\end{abstract}

\maketitle

\section{Introduction: The problem of inverse degrees}

A Cremona transformation is a birational morphism $\phi: \PP^{n} \dashrightarrow \PP^{n}$,
where we fix $n \geq 2$ to avoid trivialities and
sometimes write $\PP^n_k$ to emphasize the field $k$.
There are homogeneous polynomials $g_i$ of the same degree $d$ in variables $x_0, x_1, \dots, x_n$ such
that $\phi$ induces the partial function 
$(x_0 : \dots : x_n) \to (g_0 : \dots : g_n)\,.$  
If nonconstant factors common to all the $g_i$ are canceled, the representation of $\phi$ is unique
up to a nonzero constant factor, thus giving a well-defined degree $d$,   
even under independent coordinate changes in the domain and codomain.

The inverse of $\phi$ also has a degree, here called the inverse degree $d'$ of $\phi$, so $d'' \!=\! d$.
The notation $d'$ emphasizes that we are interested in which values of inverse degrees occur as $\phi$
ranges over Cremona transformations of degree $d$, in some fixed $\PP^n_k$.
In what follows it can be assumed that $k$ (after extending) is algebraically closed.
Classical tools, in a setting described for example in \cite{ST}, but going back to the earliest
investigations by Cremona, N\"other and Cayley, easily yield a bound $d'\leq d^{n-1}$
from the study of intersections with certain sets of hyperplanes.
In $\PP^2$ this gives $d'= d$.  In $\PP^3$ it gives $\sqrt{d} \leq d' \leq d^2$,
a fact Cremona regarded as evident --- see p.~278, l.~6  of \cite{Cr}.
Cremona \cite{Cr}, \cite{Cre} showed that all such $d'$ occur here when $d=2$ or $d=3$.
This is now established for all $d$ in Pan \cite{Pa2}, with examples that modify ones
in \cite{Pa} that were found to be flawed.
In $\PP^n$, $n > 3$, many related questions remain open.

We focus almost entirely on the much simpler monomial Cremona transformations,
where the defining polynomials $g_i$ are monomial.
These have been studied actively during the last decade --- see for example
Costa and Simis \cite{CoSi}, Simis and Villarreal \cite{SiVi}, and their references.
As described in \cite{GoPa}, with fuller details in Theorem 2.2 of \cite{SiVi},
the inverse of any monomial Cremona transformation is monomial.  Even in this special case, little is known
about which inverse degrees are possible.  Our results dash hopes for simple answers.

A recent preprint \cite{CoSi} of Costa and Simis, based on \cite{Co}, provides some definitive results
on rational maps defined by quadratic monomials.  One can see when such a map is Cremona, and if so
quickly read off the degree of its inverse, by examining a graph whose vertices correspond to the
variables, with the same number of edges (allowing loops) which correspond to the quadratic monomials. 
Note that some articles such as \cite{CoSi} use $n$ variables where we, following others, use $n+1$.
As a trivial application of the main result of \cite{CoSi}, one can see that
a monomial Cremona transformation of degree 2 in $\PP^n$, where $n \geq 2$, has inverse degree at most $n$,
with the bound attained only when the graph is either a triangle, giving the classic Cremona transformation
in $\PP^2$, or a linear tree (chain) with a loop added at one of the two extremities, giving a map
$\phi: (x_0: \dots: x_n) \to (x_0^2: x_0x_1: \dots: x_{n-1}x_n)$.
In particular, when $n=3$ and $d=2$, the maximum value of $d'$ is 3 for monomial Cremona transformations,
whereas without monomiality the maximum is 4.

An example will be provided where the monomial Cremona transformations have inverse degrees fairly close
to the upper bound of $d^{n-1}$, which in this special case ceases to be optimal for $n > 2$.
Except for very small values of $n$ and $d$, it is not practical to examine all monomial Cremona
transformations to see which inverse degrees $d'$ actually occur.
Still, the few results observed from programs reveal that the $d'$ do not always form a sequence:
gaps can appear.  This and other aspects relating to computation are discussed in the last section.

\section{Methods for obtaining inverses}

To prepare the ground for a formula for calculating inverse degrees, simplifying a method
presented in Th.~2.2 of \cite{SiVi}, we establish notation and
briefly explain easy ways to establish some of the most relevant results.

The notation used to define $\phi$ will be changed slightly to admit $m+1$ monomials
$g_0, \dots,g_m$ in $n+1$ variables, so $\phi$ becomes a rational map from $\PP^n$ to the closure
of its image in $\PP^m$.  In many sources monomials are taken to be monic, but we shall
temporarily allow constant factors, to indicate how little effect this generalization has on the maps. 
Only the main case $m=n$ will  be described, roughly following \cite{GoPa}. 
The group of monomial Cremona transformations in $\PP^n_k$, $k$ an arbitrary field,
is a split extension.  First,
the maps $(x_0 : \dots : x_n) \mapsto (c_1 x_0 : \dots : c_n x_n)$, $c_i \in k^*$,
where constant factors act trivially, form a normal subgroup $\TT$, a torus $(k^*)^n$.
Multiplying by elements of $\TT$ does not affect degrees of maps.
In the group of monomial Cremona transformations, the monic monomials, which are the real focus
of interest, form a complementary subgroup to $\TT$.
As determined in \cite{GoPa}, or from discussions about lattices below,
this subgroup can be identified with Aut$(\ZZ^n)$, or GL${}_n(\ZZ)$.
It acts on $\TT$ as the full group of rational automorphisms.

Henceforth, monomials will be monic.
Given systems of coordinates, every rational map $\phi: \PP^n \dashrightarrow \PP^m$ defined
by monomials $g_i$ can be represented by its $n+1 \times m+1$ log-matrix $A$,
whose $j^{\rm th}$ column lists the exponents of the monomial $g_j$.  
A log-matrix $A$ is stochastic or, for emphasis, $d$-stochastic:
every column has the same sum, which is $d$.
Since only exponents of the $g_j$ are used, the field $k$ is now irrelevant.
The rational map $\phi$ is unchanged if, in its log-matrix, all entries in a row are adjusted
by the same amount, remaining in $\NN$.
Thus it can usually be assumed that the minimal entry in each row of $A$ is 0, in which case
its column sums give the degree of $\phi$, and we say that $A$ is {\em reduced}.

For greater flexibility, matrix entries will now be allowed to range over $\ZZ$, 
and the $g_i$ may be Laurent monomials.
For each index $i$, let $R_i$ be the 1-stochastic matrix whose entries in the $i^{\rm th}$
row are $1$ and whose other entries are 0.
Stochastic matrices  will be called {\em equivalent} if they have the same size and their difference
is an integral linear combination of the $R_i$ or, in other words, if they induce the same rational map.
Matrices are supposed be $n+1 \times m+1$, except when monomial rational maps are composed,
which corresponds to multiplying matrices of compatible sizes, then usually passing to the reduced matrix.

To say $A$ is $d$-stochastic means ${\bf 1}A = d{\bf 1}$, where ${\bf 1}$ denotes an
all 1 row vector, here of size $n+1$.  Then multiplying by $A$ on the left induces a map of
$\ZZ$-lattices (free abelian groups) $\Lambda^m_0 \to \Lambda_0^n$, where $\Lambda^m_0$ denotes
the sublattice of $\ZZ^{m+1}$ consisting of column vectors whose entries sum to 0.
Stochastic matrices are equivalent precisely when they induce the same lattice map.

There are many criteria for deciding if a map $\phi$ is birational.
Some that are specific to the monomial situation appear in \cite{SimVi} and \cite{SiVil}. 
If the above lattice map is not an isomorphism, the associated monomial map $\phi$ has no
inverse, not even among rational maps.
A proof can be extracted from Sections 1 and 2 of \cite{SimVi}, but uses more machinery than
would be expected for such a basic fact, so the following approach seems worth recording.
The key result is:

\begin{Theorem}
If a birational map $\phi: \PP^n \dashrightarrow W$, $W$ closed in $\PP^m$, is defined by monomials,
the induced lattice map $\Lambda_0^m \to \Lambda_0^n$ must be surjective.  
\end{Theorem}

\begin{proof}
By assumption, $\phi$ has a rational inverse $\psi$, which must define an injective function
on some nonempty open subset $V$ of $W$ via $n+1$ homogeneous polynomials in variables $y_0, \dots y_m$
coordinatizing $\PP^m$, with open image $U \subset \PP^n$ on which $\phi$ is defined.
All sets here are irreducible.  If desired, one can assume that the coordinate functions $x_i$ or $y_j$
vanish nowhere on $U$ or $V$, to work with rings $R_U$, $R_V$ of functions on $U$, $V$,
generated by Laurent monomials of degree 0 in the coordinates.  Recall that
$\Lambda_0^n$ contains the image of the map on $\Lambda_0^m$ induced by the log-matrix of $\phi$.  
A suitable change of basis for $\Lambda_0^n$, and a shift to multiplicative language, produce Laurent
monomials $z_i$ that are free generators of $R_U$, such that certain powers $z_i^{d_i}$ generate the
same multiplicative group as that generated by the monomials $\phi^*(y_j)$.
Although $\psi$ induces an isomorphism between $k(U)$ and $k(V)$ (quotient fields of $R_U$ and $R_V$),
the images of the $z_i^{d_i}$ already generate $k(V)$. 
This forces $|d_i| = 1$ for all $i$, so the above lattice map is surjective. 
\end{proof}

A rational map $\phi: \PP^n \dashrightarrow \PP^n$ defined by monomials is birational precisely
when it induces an automorphism of the lattice $\Lambda_0^n$, and the inverse must be monomial.
Such maps have  been dealt with before, most notably in \cite{SiVi}.
Our main aim is to provide practical methods for obtaining inverse maps and their degrees,
where the matrices can have entries in $\ZZ$.

Let $\phi$ be birational, with $d$-stochastic log-matrix $A$ acting on $\ZZ^{n+1}$.
The restriction to $\Lambda_0^n$ has determinant $\pm 1$ and, for every column vector ${\bf v}$,
$A{\bf v} -d{\bf v}$ lies in $\Lambda_0^n$. 
Thus det$(A) = \pm d$ (cf.\ Lemma~1.2 of \cite{SimVi}).  Now assume $d \neq 0$.
Since $A^{-1}$ acts on $\Lambda_0^n$, the difference of any two columns of $A^{-1}$ is integral 
(cf.\ Lemma~2.1 of \cite{SiVi}).  This inverse is of the form $d^{-1}A^*$, where $A^*$ has entries
in $\ZZ$.  In each row $i$ of $A^{-1}$, let $r_i$ be the least element, and define $k_i = -dr_i$. 
All $k_i$ lie in $\NN$, since $A^{-1}A = I$ and $n+1 \geq 2$.  Using again the matrices $R_i$ with
$i^{\rm th}$ row ${\bf 1}$ to adjust rows, define $B = (I+\sum_i k_iR_i)A^{-1}$.  Then $B$ is the
reduced matrix that represents the inverse of $\phi$, so its column sums give the inverse degree.
This yields a formula that may have theoretical use:

\begin{Theorem}
Let $A$ be a $d$-stochastic matrix obtained from Laurent monomials that define a Cremona
transformation.  Then the inverse map is also monomial and, assuming $d \neq 0$, its degree is
$d^{-1}-\sum r_i$, where $r_i$ denotes the least value occurring in the $i^{th}$ row of $A^{-1}$.
\end{Theorem}

To clarify this result, and to eliminate the apparent need to work with fractions, 
one should work within the framework of Laurent monomials in affine coordinates $X_i = x_i/x_0$. 
We present this indirectly via an easily applicable matrix-based approach, in which $d$ and $d'$
are regarded as functions on the group $G$ of automorphisms of $\Lambda_0^n$,
with $d'(g) = d(g^{-1}) \in \ZZ^+$.

Using the ordered basis $e_i-e_0$ $(0 < i \leq n+1)$, $G$ is identified with the matrix group
GL${}_n(\ZZ)$.  Starting with the matrix of any $g \in G$, one prepends a row and column in the
zero-th position, so that all entries of this column, and all column sums, are zero. 
Adjusting rows as before yields a reduced matrix $A = A_g$ which, as should be clear from earlier
discussions, is the log-matrix of a monomial Cremona transformation $\phi$ whose column sums give
the degree $d(\phi)$.  Thus $d(g)$ is the sum of $n+1$ nonnegative terms, one for each row of $A$,
using the negative of the least term in the row, or 0 if no negative terms are present.
Similar formulas appear in Remark~3.2 of \cite{GoPa} in a loosely related context.

\medskip

\noi{\bf Example 1.} We pass back and forth between some matrices in GL${}_2(\ZZ)$ and the related
log-matrices for monomial Cremona transformations of $\PP^2$ (so $d'=d$).

\bigskip

$g = \left[\begin{array}{rr} 1& 24 \\ 1&25 \end{array}\right]
\darw\ \ \left[\begin{array}{rrr} 0& -2 &-49 \\ 0&1&24\\0&1&25 \end{array}\right]
\darw\ \ A_g \,=\, \left[\begin{array}{rrr} 49& 47&0 \\ 0&1&24\\0&1&25 \end{array}\right]$,
$\ \ d(g) = 49$.

\bigskip

{\parindent0pt

$g^{-1} = \left[\begin{array}{rr}  25 & 24 \\ -1&1 \end{array}\right]
\darw \left[\begin{array}{rrr} 0& -24 & 23 \\ 0&25&-24\\0&-1&1 \end{array}\right]
\darw A_{g^{-1}} \!=\! \left[\begin{array}{rrr} 24& 0 &47 \\ 24&49&0\\1&0&2 \end{array}\right]$,
$d(g^{-1}) = 49$.

\bigskip

$-g = \left[\begin{array}{rr}  -1& -24 \\ -1&-25 \end{array}\right]
\darw \left[\begin{array}{rrr} 0& 2 &49 \\ 0&-1&-24\\0&-1&25 \end{array}\right] \darw
A_{-g} \!=\! \left[\begin{array}{rrr} 0&2&49 \\ 24&23&0\\ 25&24&0 \end{array}\right]$,
$\,d(-g) = 49$.

\bigskip

}

$g^{\rm t} = \left[\begin{array}{rr}  1& 1 \\ 24&25 \end{array}\right]
\darw\ \left[\begin{array}{rrr} 0& -25 &-26 \\ 0&1&1\\0&24&25 \end{array}\right]
\darw\ A_{g^{\rm t}} = \left[\begin{array}{rrr} 26& 1&0 \\ 0&1&1\\0&24&25 \end{array}\right]$,
$\ \ d(g^{\rm t}) = 26$.

\bigskip

Permutations of coordinate vectors induce a subgroup $S_{n+1}$ of $G = {\rm GL}_n(\ZZ)$,
and $d$ is constant on each double coset $S_{n+1}gS_{n+1}$, since this involves
permuting rows and columns of stochastic matrices.
By considering $A_{g_1}A_{g_2}$, one sees that $d(g_1 g_2) \leq d(g_1)d(g_2)$. 
Also, for $g \notin S_{n+1}$, $d(g) \geq 2$, and $d(-I)=n$. 
Such ideas, developed further, may well yield useful information.

\section{Examples with large inverse degrees}

There are in the literature many examples of monomial Cremona transformations with diverse features
of interest --- see for example Section 5 of  \cite{SiVil}. 
An attempt to construct maps in $\PP^n$ of degree $d$ with relatively large inverse degrees $d'$
led to the following examples, which may be those with the largest possible $d'$.
Extensive tests produced no other example (up to permutations) 
that attained or exceeded the value given in this section.
We will proceed as in Theorem~2.  This should be compared with the later approach.

\medskip

\noi{\bf Example 2.}  For $d,n \geq 2$, 
define $\phi: (x_0 : \dots : x_n)
\to (x_0^d : x_0^{d-1} x_1 :  \dots : x_{n-1}^{d-1} x_n)$, which has an upper-triangular log-matrix
$A$ whose nonzero entries are $a_{11} = d$, and, for $i > 1$, $a_{i-1,i} = d-1$,  $a_{ii} = 1$.
Let $B$ be the upper-triangular matrix with entries $b_{ij} = (1-d)^{j-i}$ for $i \leq j$.
It is easy to verify that dividing the first line of $B$ by $d$ produces $A^{-1}$.
For convenience, write $c = d-1$.  
To illustrate, the matrices $A$ and $B$ (almost inverses) are shown below when $n$ is 2 or 3.
The degree of $\phi$ is $d$ and the above criterion will, after simplification, give the
inverse degree $d' = 1 + c + \dots + c^{n-1}$, which is  $\,\frac{(d-1)^{n}-1}{d-2}\,$ if $\,d > 2$.

For $n$ of the form $2m$, the raw sum for $d'$ is  $\frac{1}{c+1} + \frac{c^{2m-1}}{c+1} + c^{2m-1} +
2\sum_{i=0}^{m-2} c^{2i+1} + 0$, whose fractional parts yield an alternating sum that meshes
with the last sum.
The case $n = 2m+1$, $d' =  \frac{1}{c+1} + \frac{c^{2m+1}}{c+1} + 2\sum_{i=0}^{m-1} c^{2i+1} + 0$,
is almost identical.

\bigskip

For $n=2$,
$A = \left[\begin{array}{crr} c+1 & c & 0 \\ 0 & 1 & c \\ 0 & 0 & 1 \end{array}\right]$,
$B = \left[\begin{array}{rrr}  1 & -c & c^2 \\ 0 & 1 & -c \\ 0 & 0 & 1 \end{array}\right]$.

\bigskip

For $n=3$, 
$A = \left[\begin{array}{crrr}  c+1 & c & 0 & 0\\ 0&1&c&0\\ 0&0&1&c\\0&0&0&1 \end{array}\right]$,  
$B = \left[\begin{array}{rrrr}  1 & -c & c^2 & -c^3\\ 0 & 1 & -c & c^2\\ 0&0&1& -c \\ 0&0&0&1
\end{array}\right]$.

\section{Computational aspects}

All monomial Cremona transformations of degree 5 in $\PP^3$ can be generated from the 367,290
combinations (unordered selections) of 4 out of 56 monomials in $x_0, \dots,x_3$, discarding
those that fail to define a birational map of degree 5.  Only 11,496 maps survive.
By using combinations, the effect of permuting the columns (four coordinates in the image) has been
factored out, so that only one out of each class of 24 maps is counted.
The $S_4$ that permutes rows acts on these classes, in orbits (often regular) consisting of maps with
the same inverse degree.  The frequency of occurrence of each $d'$, but not the orbits, of these
11,496 maps, is shown below.
The $d'$ for the 48,042 monomial Cremona transformations of degree 3 in $\PP^4$ are also shown.
In that case, four kinds of maps are square-free, with descriptions given in Prop.~5.5 of \cite{SiVil}.

\bigskip
\goodbreak

\begin{center}
\capt 1;{All monomial Cremona transformations with $d=5$ in $\PP^3$.}

\begin{tabular}{|c||c|c|c|c|c|c|c|c|c|c|}
\hline
Value of $d'$: &3&4&5&6&7&8&9&10&11&12\\
Freq.\ of $d'$: &
 120 & 672 & 1932 & 1044 & 1584 & 1440 & 1248 & 696 & 816 & 552 \\
\hline
&13&14&15&16&17&18&19&20&21&\\
& 480 & 168 & 240 & 240 & 168 & 0 & 24 & 48 & 24 & \\
\hline
\end{tabular}

\bigskip

\capt 2;{All monomial Cremona transformations with $d=3$ in $\PP^4$ .}

\begin{tabular}{|c||c|c|c|c|c|c|c|}
\hline
 Value of $d'$: &2&3&4&5&6&7&8\\
Freq.\ of $d'$: &
 432 & 8670 & 14640 & 10920 & 5820 & 3720 & 1200 \\
\hline
 &9&10&11&12&13&14&15\\
 & 1080 & 840 & 360 & 0 & 240 & 0 & 120 \\
\hline
\end{tabular}

\bigskip
\end{center}

One sees from orbit sizes that, in both cases, the only examples giving the maximum value of $d'$
are those described in the previous section, up to permutations of rows and columns.
More importantly, even for these small values of $d$ and $n$, the list 
for $d'$ contains gaps instead of forming an interval. 
Similar exhaustive calculations show that for monomial Cremona transformations in $\PP^3$,
the next case $d=6$ has values of $d'$ that do form an interval $[3,31]$, and here each $d'$
from 28 to 31 arises from an essentially unique map.
When $d=7$, $d'$ can assume any value in $[3, 43]$ except 39 or 40,
while for $d=8$, $d' \in [4,57]$, excluding only 54.
A double gap first appears for $d=10$, where $d'$ cannot be $84$, $87$ or $88$.

Moving to $\PP^4$, it is unlikely that theory will help explain the results.
For example, when $d=4$, the $d'$ lie in $[2, 40]$, with gaps at 32, 34--36, and 38--39.
When $d=5$, the $d'$ lie in $[3,85]$, with gaps at 63, 70, 72, 74--75, 77--80, and 82--84.

The calculations were performed in SINGULAR \cite{DGPS}, with about 100 lines of code, available on request.
On an ordinary computer, this took a few minutes, a few hours, or (for $d=5$ in $\PP^4$) about two days.
Results from a similar program that randomly generated large numbers of Cremona transformations
with fixed $d$ and $n$ had already hinted strongly where the smaller gaps lie.

\medskip

Since this seems to be the first record of systematic calculations made in this area, it seems appropriate
to mention what data can and cannot be expected to be obtainable. 
Tables for values of $d$ up to 12 in $\PP^3$ have been produced, with 13 expected to be the last,
as the naive algorithm used becomes too slow.   An implementation to count the number of monomial
Cremona transformations in $\PP^4$ with given multidegrees \cite{Pa} $d_1 = d$, $d_2$, $d_3 = d'$,
and giving data for a few values of $d$, should soon follow.

The diversity and richness of examples obtained already provide a counterpart and challenge to theory.
To go further, it seems worthwhile to use random sampling to assemble extensive collections of certain
kinds of Cremona transformations, not just monomial ones, then search for new phenomena. 
This needs to be done with care.  We point out where some simple approaches run into difficulties.

The idea used above for random generation of monomial Cremona transformations quickly becomes useless
as $d$ and $n$ grow, as $d$-stochastic matrices giving Cremona transformations are rare,
and those giving large values of $d'$ are even rarer.
Also, the model used to generate columns ($d$ approximately independent and uniform selections
of a position) is strongly biased against outcomes where the nonzero values are concentrated in
few entries, the ones most likely to give a $d'$ that is unusually large or small.
By repeated sampling until many\footnote[1]{20,000 down to 100, as parameters grew.}
monomial Cremona transformations were obtained, the maximum $d'$ observed started to fall
well below the bound given by the examples of the previous section,
blocking attempts to clarify where further gaps might lie.
To show just one small example, by generating only 10,000 Cremona transformations with $d=6$
in $\PP^3$ we did come fairly close to the maximum value $d'= 31$ but, as typically happened,
failed to observe the minimum, here $d'= 3$, although values just above this appeared frequently.

\bigskip
\goodbreak

\begin{center}
\capt 3;{10,000 samples (from 3,226,875 attempts) with $d=6$ in $\PP^3\!.$}

\begin{tabular}{|c||c|c|c|c|c|c|c|c|}
\hline
Value of $d'$: &4&5&6&7&8&9&10&11\\
Freq.\ of $d'$: & 260 & 270 & 797 & 838 & 1169 & 1737 & 1031 & 1342 \\
\hline
  &12&13&14&15&16&17&18&19\\
  & 792 & 704 & 224 & 311 & 158 & 174 & 36 & 56 \\
\hline
 & 20&21&22&23&24&25&26&27\\
 & 64 & 17 & 3 & 6 & 1 & 9 & 0 & 1\\
\hline
\end{tabular}
\end{center}

\bigskip

A better approach is to work in a fixed $\PP^n$, using matrices in GL${}_n(\ZZ)$ to generate
monomial Cremona transformations with variable parameters $(d,d')$ (or multidegrees) whose range of
values should become clearer over time.  A stream of matrices can be obtained from the identity matrix
by making random choices to replace some line with that line plus a multiple of another line.
Any matrix in GL${}_n(\ZZ)$ is reachable by this process: just consider how to row-reduce nonsingular
matrices when working over the integers.  For a slightly different explanation, see \cite{Tr}.
In practice, best results were obtained using multiples restricted to have small absolute value
(even 5 seems too large), which helped postpone arriving at a case with $d$ or $d'$ beyond the
range considered, at which point the generation process was started afresh.
This idea has the great advantage that it produces a usable example at each step.
However, building a table of pairs $(d,d')$ that occur in a certain range is a slow process,
as the cases of most interest tend to appear rarely --- in some cases, only once in more than
330 million trials.  Several days of computation were enough to fill most of the possible positions
repeatedly, but in no case was there compelling evidence to suggest whether some new apparent gap
was real or illusory.

\bigskip

{\smc Acknowledgment.} A conversation with Ivan Pan helped clarify details about bounds for $d'$.
He also sent an early version of \cite{Pa2}.

\bigskip

\small
\frenchspacing
\bibliographystyle{plain}

\begin{thebibliography}{99} 


\bibitem{Co}{B. Costa, {\em Transforma\c c\~oes de Cremona definidas por mon\^omios},
Ph.D. thesis, UFPE, Recife, Brazil, 2011.}

\bibitem{CoSi}{B. Costa, A. Simis,
{\em Cremona maps defined by monomials}, 21 pp., arXiv:1101.2413v1 [math.AC].}

\bibitem{Cr}{L. Cremona, {\em Sulla transformazione razionale di 2.${}^\circ$ grado nello spazio,
la cui inversa \`e di 4.${}^\circ$ grado},
Memorie dell'Accademia delle Scienze dell'Istituto di Bologna, Ser. III, {\bf 1} (1871), 365--386.
\ Reprinted in \cite{Cr3}.}

\bibitem{Cre}{L. Cremona, {\em Sulle transformazioni razionali nello spazio},
Annali di Matematica pura ed applicata, ser. II, {\bf V} (1871), 131--162.
\ Reprinted in \cite{Cr3}.}

\bibitem{Cr3}{{\em Opere matematiche di Luigi Cremona}, Vol. 3, Ulrico Hoepli, Milan, 1917.}

\bibitem{DGPS} {W. Decker, G.-M. Greuel, G. Pfister, H. Sch{\"o}nemann,
{\sc Singular} {3-1-2} --- {A} computer algebra system for polynomial computations.
{http://www.singular.uni-kl.de} (2010).}

\bibitem{Fu}{W. Fulton, {\em Intersection theory}, Springer, New York, 1984.}

\bibitem{GoPa}{G. Gonzalez-Sprinberg and I. Pan, {\em On the monomial
birational maps of the projective space}, An. Acad. Brasil. Ci\^{e}nc. 
{\bf 75} (2003), no.~2, 129--134.}

\bibitem{GPa}{G. Gonzalez-Sprinberg and I. Pan, {\em On characteristic classes of determinantal
Cremona transformations}, Pr\'epublication de l'Institut Fourier n${}^{\rm o}$ \bf 681} (2005), 9 pp.


\bibitem{Pa}{I. Pan, {\em Sur le multidegr\'e des transformations de Cremona},
C. R. Acad. Sci. Paris,  S\'erie I, {\bf 330} (2000), 297--300.}

\bibitem{Pa2}{I. Pan, {\em On Cremona transformations of $\PP^3$ with all possible multidegrees},
preprint (2011), 4~pp.}

\bibitem{ST}{J. G. Semple, J. A. Tyrrell, {\em Specialization of Cremona transformations},
Mathematika {\bf 15} (1968), 171--177.}

\bibitem{SimVi}{A. Simis and R. H. Villarreal,
{\em Constraints for the normality of monomial subrings and birationality},
Proc. Amer. Math. Soc. {\bf 131} (2003), 2043--2048.}

\bibitem{SiVi}{A. Simis and R. H. Villarreal,
{\em Combinatorics of Cremona monomial maps}, 10~pp., arXiv:0904.4065v2 [math.AG],
to apppear in Math. Comp.}

\bibitem{SiVil}{A. Simis and R. H. Villarreal,
{\em Linear syzygies and birational combinatorics}, Results Math. {\bf 48} (2005), no. 3--4, 326--343.}

\bibitem{Tr}{S. Trott, {\em A pair of generators for the unimodular group}, Canad. Math. Bull. {\bf 5}
(1962), 245--252.}


\end{thebibliography}

\bigskip
{\smc\noindent
Departamento de Matem\'atica\\
Universidade Federal de Pernambuco\\
50740-540 Recife--PE\\ Brazil\\
{\tt peterj@dmat.ufpe.br}}

\end{document}